\documentclass[12pt]{amsart}
\usepackage{amssymb}
\usepackage[all]{xy}

\newtheorem{theorem}{Theorem}[section]
\newtheorem{definition}[theorem]{Definition}
\newtheorem{proposition}[theorem]{Proposition}

\begin{document}

\author{Teodor Banica}
\address{T.B.: Department of Mathematics, University of Cergy-Pontoise, F-95000 Cergy-Pontoise, France. {\tt teo.banica@gmail.com}}

\title[Super-easy quantum groups]
{Super-easy quantum groups: definition and examples}


\begin{abstract}
We investigate the ``two-parameter'' quantum symmetry groups that we previously constructed with Skalski, with the conclusion that some of these quantum groups, namely those without singletons, are ``super-easy'' in a suitable sense, that we axiomatize here. Our formalism covers as well the symplectic group $Sp_n$ and its free version $Sp_n^+$, and some other interesting examples. Finally, we address the general problem of classifying the super-easy quantum groups, and we make a few comments on it.
\end{abstract}

\maketitle

\section*{Introduction}

The easy quantum groups were introduced in \cite{bsp}, following some previous work in \cite{bbc}. The idea comes from an old paper of Brauer \cite{bra}, and from Woronowicz's Tannakian duality results in \cite{wo2}. To be more precise, a closed subgroup $G\subset U_n^+$ is called easy when its Schur-Weyl dual comes from a category of partitions. There are several known facts about such quantum groups, including a full classification result in the orthogonal case \cite{rwe}, and some partial classification results in the unitary case \cite{twe}. For some concrete applications, to questions coming from quantum physics, we refer to \cite{bcu}, \cite{bco}.

The theory is quite flexible, and several extensions ot it, not yet unified, are now available. A first series of related quantum groups, coming from noncommutative geometry considerations, appeared in our joint papers with Skalski \cite{bs1}, \cite{bs2}. Some other examples, coming from combinatorial considerations, were constructed in \cite{cwe}, \cite{fre}, \cite{swe}. Finally, some other interesting examples include the twists $O_n,O_n^{*-1}$ discussed in \cite{ban}, \cite{bbc}, and the symplectic groups and quantum groups $Sp_n,Sp_n^+$ appearing in \cite{bcz}, \cite{csn}.

Putting all this new material into a unified theory is a key open question. We will advance here on this question, with a new look at the ``two-parameter'' quantum groups constructed in \cite{bs1}, \cite{bs2}. Our main result will state that most of these quantum groups, along with $Sp_n$, and some other examples too, including most of the usual easy quantum groups, are ``super-easy'', in a certain suitable sense, that we will axiomatize here.

The classification problem for the super-easy quantum groups looks quite difficult. A first non-trivial task is that of understanding what are the supplementary examples coming from the above-mentioned work in \cite{cwe}, \cite{fre}, \cite{swe}. As for the classification work itself, this would probably require the heavy use of a computer, or at least this is our belief.

The paper is organized as follows: 1 is a preliminary section, in 2-3 we review our work on the two-parameter quantum symmetry groups, and in 4 we axiomatize the super-easy quantum groups, and we comment on the classification problem for them.

\medskip

\noindent {\bf Acknowledgements.} I would like to thank Adam Skalski for useful discussions.

\section{The super space}

We use the quantum group formalism of Woronowicz \cite{wo1}, \cite{wo2}, under the supplementary assumption $S^2=id$. In order to introduce these quantum groups, best is to recall that any compact Lie group appears as a closed subgroup $G\subset U_n$. Let us recall as well that $U_n$ has a free analogue $U_n^+$, constructed by Wang in \cite{wan}, as follows:

\begin{definition}
$U_n^+$ is abstract spectrum of the universal Hopf $C^*$-algebra
$$C(U_n^+)=C^*\left((u_{ij})_{i,j=1,\ldots,n}\Big|u^*=u^{-1},u^t=\bar{u}^{-1}\right)$$
with $\Delta(u_{ij})=\sum_ku_{ik}\otimes u_{kj}$, $\varepsilon(u_{ij})=\delta_{ij}$, $S(u_{ij})=u_{ji}^*$ as structural maps.
\end{definition}

Observe that $C(U_n^+)$ satisfies indeed Woronowicz's axioms in \cite{wo1}, \cite{wo2}, along with the extra axiom $S^2=id$. The point now is that the quantum groups axiomatized in \cite{wo1}, \cite{wo2}, under the extra axiom $S^2=id$, coincide with the closed quantum subgroups $G\subset U_n^+$. These will be the objects that we will be interested in, in what follows. 

For more details on the general theory here, we refer to \cite{ntu}, \cite{wo1}, \cite{wo2}.

We have the following key result, coming from \cite{bdv}, \cite{vwa}:

\begin{theorem}
Given a closed subgroup $G\subset U_N^+$, its fundamental corepresentation is self-adjoint, $u\sim\bar{u}$, precisely when $u=J\bar{u}J^{-1}$ for some linear map $J:\mathbb C^n\to\mathbb C^n$ satisfying $JJ^*=1,J\bar{J}=\pm 1$. Moreover, up to an orthogonal base change, we can assume
$$J=\begin{pmatrix}
0&1\ \ \ \\
\varepsilon 1&0_{(0)}\\
&&\ddots\\
&&&0&1\ \ \ \\
&&&\varepsilon 1&0_{(p)}\\
&&&&&1_{(1)}\\
&&&&&&\ddots\\
&&&&&&&1_{(q)}
\end{pmatrix}$$
where $n=2p+q$ and $\varepsilon=\pm 1$, with the $1_q$ block at right disappearing if $\varepsilon=-1$.
\end{theorem}

\begin{proof}
We follow \cite{bdv}, where an analogue of this result was proved in the general, non $S^2=id$ case. First, we have $u=J\bar{u}J^{-1}$, with $J\in GL_n(\mathbb C)$. By conjugating, $\bar{u}=\bar{J}u\bar{J}^{-1}$, so $u=(J\bar{J})u(J\bar{J})^{-1}$, and since $u$ is assumed irreducible, $J\bar{J}=c1$ with $c\in\mathbb C$. By conjugating, $\bar{J}J=\bar{c}1$, and hence $c\in\mathbb R$. By rescaling we may assume $c=\pm 1$.

Since we are in the $S^2=id$ case we have $S*=*S$, so from $(id\otimes S)u=u^*$ we get $(id\otimes S)\bar{u}=u^t$. By applying $id\otimes S$ to $u=J\bar{u}J^{-1}$ we get $u^*=Ju^tJ^{-1}$, so $u=(J^*)^{-1}\bar{u}J^*$, so $\bar{u}=J^*u(J^*)^{-1}$. With $u=J\bar{u}J^{-1}$ this gives $JJ^*\in End(u)$, so $JJ^*=d1$ with $d\in\mathbb C$.

We have $JJ^*>0$, so $d>0$. From $J\bar{J}=\pm 1$ and $JJ^*=d1$ we get $|\det J|^2=(\varepsilon)^n=d^n$, so $d=1$, and if $n$ is odd we must have $\varepsilon=1$. This ends the proof of the first assertion. As for the second assertion, the proof here is elementary, and can be found in \cite{bdv}.
\end{proof}

The above result suggests the following definition:

\begin{definition}
The ``super-space'' $\bar{\mathbb C}^n$ is the usual space $\mathbb C^n$, with its standard basis $\{e_1,\ldots,e_n\}$, with a chosen sign $\varepsilon=\pm 1$, and a chosen involution $i\to\bar{i}$ on the set of indices. The ``super-identity'' matrix is $J_{ij}=\delta_{i\bar{j}}$ for $i\leq j$ and $J_{ij}=\varepsilon\delta_{i\bar{j}}$ for $i\geq j$.
\end{definition}

In what follows we will usually assume that $J$ is the matrix appearing in Theorem 1.2. Indeed, up to a permutation of the indices, we have a decomposition $n=2p+q$ such that the involution is $(12)\ldots (2p-1,2p)(2p+1)\ldots (q)$, and this gives the above matrix $J$.

Thus in the case $\varepsilon=1$, the super-identity is:
$$J_+(p,q)=\begin{pmatrix}
0&1\ \ \ \\
1&0_{(1)}\\
&&\ddots\\
&&&0&1\ \ \ \\
&&&1&0_{(p)}\\
&&&&&1_{(1)}\\
&&&&&&\ddots\\
&&&&&&&1_{(q)}
\end{pmatrix}$$

In the case $\varepsilon=-1$ now, the diagonal terms vanish, and the super-identity is:
$$J_-(p,0)=\begin{pmatrix}
0&1\ \ \ \\
-1&0_{(1)}\\
&&\ddots\\
&&&0&1\ \ \ \\
&&&-1&0_{(p)}
\end{pmatrix}$$

Let us construct now some basic compact groups, in our ``super'' setting:

\begin{definition}
Associated to $\bar{\mathbb C}^n$ are the following compact groups:
\begin{enumerate}
\item The super-orthogonal group, $\bar{O}_n=\{U\in U_n|U=J\bar{U}J^{-1}\}$.

\item The super-symmetric group $\bar{S}_n=\{U\in\bar{O}_n|U_{ij}\in\{0,1\}\}$.

\item The super-hyperoctahedral group $\bar{H}_n=\{U\in\bar{O}_n|U_{ij}\in\{0,\mathbb T\}\}$.

\item The super-bistochastic group $\bar{B}_n=\{U\in\bar{O}_n|\sum_iU_{ij}=\sum_jU_{ij}=1\}$.
\end{enumerate}
\end{definition}

Observe that in the case $J=id$ we obtain the groups $O_n,S_n,H_n,B_n$. In general, the fact that the matrices in (1-4) above form groups is clear from definitions.

We can construct as well some basic quantum groups, as follows:

\begin{definition}
Associated to $\bar{\mathbb C}^n$ are the following Hopf algebras:
\begin{enumerate}

\item $C(\bar{O}_n^+)=C^*((u_{ij})_{i,j=1,\ldots,n}|u=J\bar{u}J^{-1}={\rm unitary})$.

\item $C(\bar{S}_n^+)=C(\bar{O}_n^+)/<u_{ij}={\rm projections}>$.

\item $C(\bar{H}_n^+)=C(\bar{O}_n^+)/<u_{ij}={\rm partial\ isometries}>$.

\item $C(\bar{B}_n^+)=C(\bar{O}_n^+)/<\sum_iu_{ij}=\sum_ju_{ij}=1>$.
\end{enumerate}
\end{definition}

In the case $J=id$ the quantum groups that we obtain are the quantum groups $O_n^+,S_n^+,H_n^+,B_n^+$ from \cite{bsp}. More generally, at $J=J_+(p,q)$ with $p,q\in\mathbb N$ we obtain the quantum groups $O^+(p,q),S^+(p,q),H^+(p,q),B^+(p,q)$ constructed in \cite{bs1}.

Observe that, by Theorem 1.2, the class of algebras $C(\bar{O}_n^+)$ coincides with the class of algebras $A_o(F)$ introduced by Van Daele and Wang in \cite{vwa}, under the assumptions that we are in the $S^2=id$ case, and the fundamental corepresentation is irreducible.

\section{Structural results}

In this section we discuss the general algebraic structure of the ``basic'' groups and quantum groups, constructed in Definition 1.4 and Definition 1.5 above.

We will need the following definition, which is standard:

\begin{definition}
Given $n\in 2\mathbb N$ we let $J=J_-(p,0)$, with $p=n/2$, and we define:
\begin{enumerate}
\item The symplectic group, $Sp_n=\{U\in U_N|U=J\bar{U}J^{-1}\}$.

\item The Hopf algebra $C(Sp_n^+)=C^*((u_{ij})_{i,j=1,\ldots,n}|u=J\bar{u}J^{-1}={\rm unitary})$.
\end{enumerate}
\end{definition}

Observe that $Sp_n,Sp_n^+$ are particular cases of the quantum groups $\bar{O}_n,\bar{O}_n^+$ constructed above. The combinatorics of these quantum groups, which partly motivates the present considerations, was investigated in a number of papers, including \cite{bcz}, \cite{bki}, \cite{csn}.

We will need as well a number of standard operations, namely the direct product $\times$, the wreath product $\wr$ and the free dual product $\hat{*}$, for which we refer to \cite{bbc}, \cite{wan}.

With these notions in hand, we have the following result:

\begin{theorem}
The basic groups and quantum groups are as follows:
\begin{enumerate}
\item At $\varepsilon=1$ we have $\bar{O}_n=O_n$, $\bar{S}_n=H_p\times S_q$, $\bar{H}_n=(\mathbb T^p\wr H_p)\times H_q$, $\bar{B}_n=O_{n+\delta-2}$ and $\bar{O}_n^+=O_n^+$, $\bar{S}_n^+=H_p^+\;\hat{*}\;S_q^+$, $\bar{B}_n^+=O_{n+\delta-2}^+$, where $\delta=\delta_{pq,0}$.

\item At $\varepsilon=-1$ we have $\bar{O}_n=Sp_n$, $\bar{S}_n=S_p$, $\bar{H}_n=\mathbb T^p\wr H_p$, $\bar{B}_n=Sp_{n-2}$ and $\bar{O}_n^+=Sp_n^+$, $\bar{S}_n^+=S_p^+$, $\bar{B}_n^+=Sp_{n-2}^+$.
\end{enumerate}
\end{theorem}

\begin{proof}
We first prove the assertions regarding $\bar{O}_n,\bar{O}_n^+$. At $\varepsilon=-1$ these follow from definitions. At $\varepsilon=1$ now, consider the root of unity $\rho=e^{2\pi i/8}$, and let:
$$\Gamma=\frac{1}{\sqrt{2}}\begin{pmatrix}\rho&\rho^7\\ \rho^3&\rho^5\end{pmatrix}$$

Then $\Gamma$ is unitary and $\Gamma(^0_1{\ }^1_0)\Gamma^t=1$, so $C=diag(\Gamma^{(1)},\ldots,\Gamma^{(p)},1_q)$ is unitary as well, and satisfies $CJC^t=1$. Thus in terms of $V=CUC^*$ the relations $U=J\bar{U}J^{-1}=$ unitary defining $\bar{O}_n^+$ simply read $V=\bar{V}=$ unitary, so we obtain an isomorphism $\bar{O}_n^+=O_n^+$ as in the statement. By passing to classical versions, we obtain as well $\bar{O}_n=O_n$. See \cite{bs1}.

The assertions regarding $\bar{B}_n,\bar{B}_n^+$ are proved in \cite{bs1} at $\varepsilon=1$, and the proof at $\varepsilon=-1$ is similar. Indeed, if we denote by $\xi$ the column vector filled with 1's, then we have $\bar{B}_n=\{U\in Sp_n|U\xi=\xi\}$. Now since for $U\in Sp_n$ we have $U\xi=\xi$ if and only if $UJ\xi=J\xi$, we conclude that the elements $U\in\bar{B}_n$ act on $\mathbb R^n=(\mathbb R\xi\oplus\mathbb R J\xi)\oplus\mathbb R^{n-2}$ by the identity on the first summand, and by a symplectic matrix on the second summand, so $\bar{B}_n=Sp_{n-2}$. Also, by using quantum groups instead of groups, we get $\bar{B}_n^+=Sp_{n-2}^+$.

We will compute now the group $\bar{H}_n$, and then its subgroup $\bar{S}_n$.

(1) With $J=J_+(p,q)$, the matrices $J\bar{U},UJ$ are obtained from $\bar{U},U$ by interchanging the rows (resp. columns) $i,i+1$, with $i=1,3,\ldots,2p-1$. Thus $J\bar{U}=UJ$ reads:
$$U=\begin{pmatrix}
a&b&\ldots&v\\
\bar{b}&\bar{a}&\ldots&\bar{v}\\
\ldots&\ldots&\ldots&\ldots\\
w&\bar{w}&\ldots&X
\end{pmatrix}$$

Here $a,b,\ldots\in\mathbb C$, $v,\ldots\in\mathbb C^q$ are row vectors, $w,\ldots\in\mathbb C^q$ are column vectors, and $X\in M_q(\mathbb R)$. In the case $U\in\bar{H}_n$, as only one entry in each row/column may be non-zero, the vectors $v,\ldots$ and $w,\ldots$ must vanish, we must have $X\in H_q$, and all $(^a_{\bar{b}}{\ }^b_{\bar{a}})$ blocks must be either $(^0_0{\ }^0_0)$ or the following form, with $z\in\mathbb T$:
$$\begin{pmatrix}z&0\\ 0&\bar{z}\end{pmatrix}\quad
\begin{pmatrix}0&z\\ \bar{z}&0\end{pmatrix}$$

Thus we obtain $\bar{H}_n=(\mathbb T^p\wr H_p)\times H_q$, and then $\bar{S}_n=H_p\times S_q$, as claimed.

(2) With $J=J_-(p,0)$ where $p=n/2$, the relation $J\bar{U}=UJ$ reads:
$$U=\begin{pmatrix}a&b&\ldots\\
-\bar{b}&\bar{a}&\ldots\\
\ldots&\ldots&\ldots
\end{pmatrix}$$

Here $a,b,\ldots\in\mathbb C$. In the case $U\in\bar{H}_n$, the $(^{\ a}_{-\bar{b}}{\ }^b_{\bar{a}})$ blocks can be either $(^0_0{\ }^0_0)$ or of the following form, with $z\in\mathbb T$:
$$\begin{pmatrix}z&0\\ 0&z\end{pmatrix}\quad
\begin{pmatrix}0&z\\ -\bar{z}&0\end{pmatrix}$$

By using once again the orthogonality condition, we must have exactly one nonzero block on each row and column, so we get $\bar{H}_n=\mathbb T^p\wr H_p$, and then $\bar{S}_n=S_p$, as claimed.

Finally, the proof of the remaining assertions, regarding $\bar{S}_n^+$, is quite similar: at $\varepsilon=1$ this is done in \cite{bs1}, and at $\varepsilon=-1$ the same arguments apply.
\end{proof}

\section{Representation theory}

In what follows we discuss the representation theory of the ``basic'' groups and quantum groups, from Definition 1.4 and Definition 1.5, with the aim of reaching in this way to a general axiomatic framework, extending the easy quantum group theory.

This is a quite non-trivial task, for several reasons. Leaving the full discussion for later, in section 4 below, let us mention that a first problem, already known from the twisting considerations in \cite{ban}, comes from the singletons. As in \cite{ban}, we will restrict here the attention to the situation where our partitions have an even number of blocks:

\begin{definition}
We let $P_{even}(k,l)$ be the set of partitions between an upper row of $k$ points, and a lower row of $l$ points, with all blocks having even size.
\end{definition}

Observe that when $k+l$ is odd, we have by definition $P_{even}(k,l)=\emptyset$. When $k+l$ is even, we denote by $1_{k,l}\in P_{even}(k,l)$ the one-block partition.

Let $J:\mathbb C^n\to\mathbb C^n$ be as in Definition 1.3. For any $i\in\{1,\ldots,n\}$ we write $J(e_i)=\varepsilon(i)e_{\bar{i}}$, with $\varepsilon(i)\in\{-1,1\}$. Observe that, by using $J^2=\varepsilon$, we obtain $\varepsilon(i)\varepsilon(\bar{i})=\varepsilon$.

We have the following definition, inspired from \cite{ban}, \cite{bcz}, \cite{bs1}, \cite{bsp}:

\begin{definition}
Associated to $\pi\in P_{even}(k,l)$ is the linear map $T_\pi:(\bar{\mathbb C}^n)^{\otimes k}\to(\bar{\mathbb C}^n)^{\otimes l}$,
$$T_\pi(e_{i_1}\otimes\ldots\otimes e_{i_k})=\sum_{j_1\ldots j_l}\delta_\pi\begin{pmatrix}i_1&\ldots&i_k\\ j_1&\ldots&j_l\end{pmatrix}e_{j_1}\otimes\ldots\otimes e_{j_l}$$
where $\delta_\pi=\prod_{b\in\pi}\delta_b$, with the product over all blocks of $\pi$, and
$$\delta_{1_{k,l}}\begin{pmatrix}i_1&\ldots&i_k\\ j_1&\ldots&j_l\end{pmatrix}=\varepsilon(i_1)\varepsilon(i_3)\ldots\times\varepsilon(j_1)\varepsilon(j_3)\ldots\times\delta_{alt}(j_1,\ldots,j_l,\bar{i}_k,\ldots,\bar{i}_1)$$
where $\delta_{alt}(l)=1$ if $l_1=\bar{l}_2=l_3=\bar{l}_4=\ldots$, and $\delta_{alt}(l)=0$ otherwise.
\end{definition}

Our first task is to check that the usual categorical operations on the linear maps $T_\pi$, namely the composition, tensor product and conjugation, are compatible with the usual categorical operations on the partitions $\pi$, namely the vertical and horizontal concatenation $[^\sigma_\pi]$ and $\pi\sigma$, and the upside-down turning $\pi^*$. The result here is as follows:

\begin{proposition}
We have the following formulae
$$T_\pi T_\sigma=\varepsilon^{c(\sigma,\pi)}n^{d(\sigma,\pi)}T_{[^\sigma_\pi]},\quad T_\pi\otimes T_\sigma=T_{\pi\sigma},\quad T_\pi^*=T_{\pi^*},\quad T_|=id$$
where $c(\sigma,\pi), d(\sigma,\pi)$ are certain positive integers.
\end{proposition}

\begin{proof}
By using the definition of $\pi\to T_\pi$, we just have to understand the behaviour of the generalized Kronecker symbol operation $\pi\to\delta_\pi$, under the various categorical operations on the partitions $\pi$. Regarding the vertical concatenation, our claim is that:
$$\sum_{j_1\ldots j_L}\delta_\sigma\begin{pmatrix}i_1&\ldots& i_k\\ j_1&\ldots&j_l\end{pmatrix}\delta_\pi\begin{pmatrix}j_1&\ldots& j_l\\ J_1&\ldots&J_L\end{pmatrix}=\varepsilon^{c(\sigma,\pi)}n^{d(\sigma,\pi)}\delta_{[^\sigma_\pi]}\begin{pmatrix}i_1&\ldots& i_k\\ J_1&\ldots&J_L\end{pmatrix}$$

Indeed, the numbers $c(\sigma,\pi)$ and $d(\sigma,\pi)$ basically count the various types of indices $x\in\{1,\ldots,L\}$, when doing the vertical concatenation operation: with respect to this operation, the indices $x$ can belong to closed or open blocks, having odd or even size, and this leads to the above formula, involving multiplicities $c(\sigma,\pi),d(\sigma,\pi)$.

With this claim in hand, the first equality follows. Regarding now the second equality, this simply follows from $\delta_\pi=\prod_{b\in\pi}\delta_b$, because we have:
$$\delta_\pi\begin{pmatrix}i\\ j\end{pmatrix}
\delta_\sigma\begin{pmatrix}I\\ J\end{pmatrix}
=\delta_{\pi\sigma}\begin{pmatrix}i\ I\\ j\ J\end{pmatrix}$$

Regarding now the third equality, this follows from the following formulae:
\begin{eqnarray*}
\delta_{1_{k,l}}\begin{pmatrix}i_1&\ldots&i_k\\ j_1&\ldots&j_l\end{pmatrix}
&=&\varepsilon(i_1)\varepsilon(i_3)\ldots\times\varepsilon(j_1)\varepsilon(j_3)\ldots\times\delta_{alt}(j_1,\ldots,j_l,\bar{i}_k,\ldots,\bar{i}_1)\\
\delta_{1_{l,k}}\begin{pmatrix}j_1&\ldots& j_l\\ i_1&\ldots&i_k\end{pmatrix}
&=&\varepsilon(j_1)\varepsilon(j_3)\ldots\times\varepsilon(i_1)\varepsilon(i_3)\ldots\times\delta_{alt}(i_1,\ldots,i_k,\bar{j}_l,\ldots,\bar{j}_1)
\end{eqnarray*}

Indeed, these numbers are equal, and by using $\delta_\pi=\prod_{b\in\pi}\delta_b$ we obtain $T_\pi^*=T_{\pi^*}$.

Finally, regarding the last assertion, observe that we have:
$$\delta_|(^i_j)=\varepsilon(i)\varepsilon(j)\delta_{alt}(j,\bar{i})=\varepsilon(i)\varepsilon(j)\delta_{ij}=\delta_{ij}$$

But this gives the last formula, and we are done.
\end{proof}

We denote by $P_2\subset P_{even}$ the category of pairings, and by $NC_2\subset NC_{even}\subset P_{even}$ the categories of noncrossing pairings, respectively of all noncrossing partitions having even blocks. See \cite{bsp}. With this convention, we have the following result:

\begin{theorem}
For the groups $\bar{O}_n,\bar{H}_n$ and the quantum groups $\bar{O}_n^+,\bar{H}_n^+$ we have
$$Hom(u^{\otimes k},u^{\otimes l})=span\left(T_\pi\Big|\pi\in D(k,l)\right)$$
for any $k,l\in\mathbb N$, where $D$ is respectively $P_2,P_{even}$ and $NC_2,NC_{even}$.
\end{theorem}

\begin{proof}
We must prove that we have a correspondence between quantum groups and categories of partitions as follows, with all the arrows standing for inclusions:
$$\xymatrix@R=40pt@C=48pt
{\bar{H}_n^+\ar[r]&\bar{O}_n^+\\
\bar{H}_n\ar[r]\ar[u]&\bar{O}_n\ar[u]}\qquad
\xymatrix@R=20pt@C=35pt{\\ :\\}\quad
\xymatrix@R=42pt@C=40pt
{NC_{even}\ar[d]&NC_2\ar[d]\ar[l]\\
P_{even}&P_2\ar[l]}$$

We already know from Proposition 3.3 that the linear spaces $span(T_\pi|\pi\in P(k,l))$ form a tensor $C^*$-category in the sense of Woronowicz \cite{wo2}, so what we have to do now is to prove that the categories on the right produce indeed the quantum groups on the left.

In order to do so, we proceed as in \cite{bsp}, and in subsequent papers. First of all, regarding $\bar{O}_n,\bar{O}_n^+$, and the categories of pairings $P_2,NC_2$, here the formula in Definition 3.2 above is precisely the formula in \cite{bcz}, and the computation there gives the result.

Regarding now $\bar{H}_n,\bar{H}_n^+$, here it is enough to prove the result for $\bar{H}_n^+$. Indeed, since $\bar{H}_n$ is the classical version of $\bar{H}_n^+$, obtained at the level of the corresponding Hopf algebras by dividing by the commutator ideal, at the Tannakian level the passage $\bar{H}_n^+\to\bar{H}_n$ is obtained by adding to the category the standard crossing, corresponding to the commutation relations $ab=ba$. But from $NC_{even}$ we obtain in this way $P_{even}$, as claimed.

So, let us discuss now the computation for $\bar{H}_n^+$. We recall from Definition 1.5 above that this quantum group has the following presentation:
$$C(\bar{H}_n^+)=C(\bar{O}_n^+)\Big/\Big<aa^*a=a,\forall a\in\{u_{ij}\}\Big>$$

As explained in \cite{bs1}, the partial isometry relations $aa^*a=a$, when combined with the biunitarity of the fundamental corepresentation $u=(u_{ij})$, show that the standard coordinates $u_{ij}$ satisfy $ab=ba$, for any $a,b$ distinct on the same row or column of $u$.

We conclude that the defining relations for $\bar{H}_n^+$ can be written as follows:
$$\sum_{abc}u_{ai}u_{\bar{b}i}^*u_{ci}\otimes e_a\otimes e_{\bar{b}}\otimes e_c=\sum_au_{ai}\otimes e_a\otimes e_{\bar{a}}\otimes e_a,\forall i$$

Now observe that from $u=J\bar{u}J^{-1}$ we obtain succesively:
\begin{eqnarray*}
u=J\bar{u}J^{-1}
&\implies&uJ(e_j)=J\bar{u}(e_j),\forall j\\
&\implies&\sum_i\varepsilon(j)u_{\bar{i}\,\bar{j}}\otimes e_{\bar{i}}=\sum_i\varepsilon(i)u_{ij}^*\otimes e_{\bar{i}},\forall j\\
&\implies&u_{ij}^*=\varepsilon(i)\varepsilon(j)u_{\bar{i}\,\bar{j}},\forall i,j
\end{eqnarray*}

By using this formula, we can write our defining relations as follows:
$$\sum_{abc}\varepsilon(b)u_{ai}u_{b\bar{i}}u_{ci}\otimes e_a\otimes e_{\bar{b}}\otimes e_c=\sum_a\varepsilon(i)u_{ai}\otimes e_a\otimes e_{\bar{a}}\otimes e_a,\forall i$$

Now observe that this equality tells us precisely that we must have $u^{\otimes 3}T=Tu$, and so that we must have $T\in Hom(u,u^{\otimes 3})$, where $T$ is the following linear map:
$$T:e_i\to\varepsilon(i)e_i\otimes e_{\bar{i}}\otimes e_i$$

On the other hand, according to Definition 3.2 we have $T=T_\pi$, with $\pi=1_{1,3}$. Now since this partition generates the whole category $NC_{even}$, we obtain the result.
\end{proof}

\section{Super-easiness}

Our aim here is to put the results obtained above in an axiomatic setting. Together with \cite{bsp}, these results suggest the following definition:

\begin{definition}
A quantum group $\bar{H}_n\subset G\subset\bar{O}_n^+$ is called ``super-easy'' if
$$Hom(u^{\otimes k},u^{\otimes l})=span\left(T_\pi\Big|\pi\in D(k,l)\right)$$
for a certain category of partitions $NC_2\subset D\subset P_{even}$.
\end{definition}

At the level of examples, besides the $2+2$ basic series from Theorem 3.4 we have, as in the usual easy case \cite{bsp}, half-liberations of the groups $\bar{O}_n,\bar{H}_n$, obtained by imposing the ``half-commutation'' relations $abc=cba$ to the standard generators $u_{ij}$:

\begin{proposition}
The following quantum groups are super-easy:
\begin{enumerate}
\item $\bar{O}_n^*$, given by $C(\bar{O}_n^*)=C(\bar{O}_n^+)/<abc=cba>$.

\item $\bar{H}_n^*$, obtained as $\bar{H}_n^*=\bar{H}_n^+\cap\bar{O}_n^*$.
\end{enumerate}
\end{proposition}

\begin{proof}
This follows from the fact that the partition $\slash\hskip-2mm\backslash\hskip-1.7mm|\hskip1mm$, implementing the half-commutation relations $abc=cba$, fits into the super-easy framework. Indeed, we have:
$$\delta_{\slash\hskip-1.7mm\backslash\hskip-1.1mm|\hskip0.5mm}\begin{pmatrix}a&b&c\\i&j&k\end{pmatrix}=\delta_|\begin{pmatrix}a\\k\end{pmatrix}\delta_|\begin{pmatrix}b\\j\end{pmatrix}\delta_|\begin{pmatrix}c\\i\end{pmatrix}=\delta_{ak}\delta_{bj}\delta_{ci}$$

We deduce that we have $T_{\slash\hskip-1.7mm\backslash\hskip-1.1mm|\hskip0.5mm}(e_a\otimes e_b\otimes e_c)=e_c\otimes e_b\otimes e_a$, so the same computation as in the usual easy case \cite{bsp} applies, and shows that the condition $T_{\slash\hskip-1.7mm\backslash\hskip-1.1mm|\hskip0.5mm}\in End(u^{\otimes 3})$ is equivalent to the half-commutation relations $abc=cba$ between the standard generators $u_{ij}$.
\end{proof}

Our notion of super-easiness is certainly not the most general one, because it still does not cover $\bar{S}_n,\bar{B}_n,\bar{S}_n^+,\bar{B}_n^+$, nor the unitary easy quantum groups from \cite{twe}, nor several interesting examples in the orthogonal case, such as the twists $O_n^{-1},O_n^{*-1}$ from \cite{ban}.

Leaving aside $\bar{S}_n,\bar{B}_n,\bar{S}_n^+,\bar{B}_n^+$ and other quantum groups of the same type, these supplementary examples suggest the following extension of Definition 4.1:

\begin{definition}
A quantum group $G\subset U_n^+$ is called ``weakly super-easy'' if
$$Hom(u^{\otimes k},u^{\otimes l})=span\left(T_\pi^\delta\Big|\pi\in\mathcal D(k,l)\right)$$
for a certain category of colored partitions $\mathcal D\subset\mathcal{P}_{even}$, where $\mathcal P_{even}$ is the category of partitions with even blocks, and legs labelled black and white, and where
$$T_\pi^\delta(e_{i_1}\otimes\ldots\otimes e_{i_k})=\sum_{j_1\ldots j_l}\delta_\pi\begin{pmatrix}i_1&\ldots&i_k\\ j_1&\ldots&j_l\end{pmatrix}e_{j_1}\otimes\ldots\otimes e_{j_l}$$
for a certain choice of the Kronecker symbols $\delta_\pi\in\{-1,0,1\}$.
\end{definition}

To be more precise, we use here the formalism in \cite{twe}, under the assumption that the blocks have even size. Our point comes from the fact that when allowing the Kronecker function $\delta$ to be signed and arbitrary, we can cover in this way the quantum groups from Definition 4.1, as well as other examples, such as the twists $O_n^{-1},O_n^{*-1}$ from \cite{ban}.

The problem of classifying such quantum groups is open. On one hand we have the problem of classifying the categories of partitions $\mathcal D\subset\mathcal{P}_{even}$, which is well-known and difficult, and on the other hand we have the problem of classifying the possible Kronecker functions $\delta$, which is non-trivial as well, and that we would like to raise here.

\end{document}